\documentclass[a4paper,11pt]{article}

\usepackage{amsthm, amssymb, amsmath}
\usepackage{algorithmic, algorithm}

\newtheorem{thm}{Theorem}[section]
\newtheorem{lemma}[thm]{Lemma}
\newtheorem{coro}[thm]{Corollary}
\newtheorem{defn}[thm]{Definition}
\newtheorem{prop}[thm]{Proposition}

\newcommand{\N}{\mathbb{N}}
\newcommand{\Z}{\mathbb{Z}}
\newcommand{\Q}{\mathbb{Q}}
\newcommand{\R}{\mathbb{R}}
\newcommand{\C}{\mathbb{C}}
\newcommand{\F}{\mathbb{F}}
\newcommand{\M}{\mathfrak{M}}

\newcommand{\vars}{[X_1^{\pm 1},\ldots,X_n^{\pm 1}]}

\title{Multivariate ultrametric root counting}
\author{Mart\'\i n Avenda\~no, Ashraf Ibrahim}

\begin{document}

\maketitle

\begin{abstract}
Let $K$ be a field, complete with respect to a discrete non-archimedian valuation and let $k$ be the residue field.
Consider a system $F$ of $n$ polynomial equations in $K\vars$. Our first result is a reformulation of the classical Hensel's
Lemma in the language of tropical geometry: we show sufficient conditions (semiregularity at $w$) that guarantee
that the first digit map $\delta:(K^\ast)^n\to(k^\ast)^n$ is a one to one correspondence between the solutions of
$F$ in $(K^\ast)^n$ with valuation~$w$ and the solutions in $(k^\ast)^n$ of the initial form system ${\rm in}_w(F)$.
Using this result, we provide an explicit formula for the number of solutions in $(K^\ast)^n$ of a certain class of
systems of polynomial equations (called regular), characterized by having finite tropical prevariety, by having initial forms
consisting only of binomials, and by being semiregular at any point in the tropical prevariety. Finally, as a
consequence of the root counting formula, we obtain the expected number of roots in $(K^\ast)$ of univariate
polynomials with given support and random coefficients.
\end{abstract}

\section{Introduction}\label{sec-intro}

The problem of counting the number of roots of univariate polynomials has been studied
for at least 400 years. The first result that we point out here, stated by Descartes
in 1637~\cite{Desca}, says that the number of positive roots (counted with multiplicities) of a nonzero
polynomial $f\in\R[x]$ is bounded by the number of sign alternations in the sequence of
coefficients of $f$. Over the reals, the problem of root counting was finally solved by
Sturm in 1829, who gave a simple algebraic procedure to determine the exact (as opposed to
an upper bound) number of real roots of a polynomial~$f$ in a given interval~$[a,b]$.
The problem was consider settled for many years until a interest in sparse polynomials
began to grow. While Sturm's technique can count the exact number of roots of any polynomial,
it is highly inefficient for polynomials of high degree with only a few nonzero terms, and
also failed to provide any insight on the roots of such polynomials. On the other hand,
Descartes' rule seems to be more natural for highly sparse polynomials: a simple consequence
of the rule is that the number of nonzero real roots of a polynomial is bounded by twice the
number of its nonzero terms. Incidentally, it has been discovered recently (see~\cite{Ave00})
how to make Descartes' rule count the exact number of real roots: the trick is to multiply
the polynomial by a high enough power of $x+1$ before counting the sign alternations.
Unfortunately, this procedure destroys completely the sparseness of the input polynomial.

\bigskip

In our search for a similar result over different fields, we decided to focus our attention
to complete fields with respect to a non-arquimedian valuation. There were several results
in this setting that indicate that an efficient root counting technique was feasible for
these fields. The first of those results, obtained by H.W.~Lenstra in 1999~\cite{Len99},
gives an upper bound for the number of nonzero roots in~$\Q_p$ (the field of $p$-adic
numbers) of a polynomial $f\in\Q_p[x]$ as a function of the number of nonzero terms of~$f$.
The second, obtained by B.~Poonen in 1998~\cite{Poo98}, gives a similar bound over~$\F_p((u))$
(the field of formal Laurent series with coefficients in~$\F_p$). Using a more unifying
approach, more of these upper bounds for ordered fields, finite extensions of~$\Q_p$, and
Laurent series with coefficients in fields of characteristic zero, were obtained by
M.~Avenda\~no and T.~Krick in 2011~\cite{AvK11}.

\bigskip

In a previous paper (see~\cite{AvIbr}), we showed a root counting procedure for univariate
polynomials that do not destroy the sparsity of the given polynomial. The technique uses a
combination of Hensel's Lemma and Newton Polygon to reduce root counting to solving binomials
over the residue field. The only drawback of this result is that it works only with
regular polynomials, which is an extensive class of polynomials defined in that paper, but
not for generic polynomials in the usual sense. In this paper, we succeded to extend those
results (root counting procedure and upper bounds) to the multivariate setting, to provide a
better understanding of the size of the class of regular polynomials, and also estimates for
the expected number of zeros of random sparse polynomials. Our counting procedure uses basic
tropical geometry and a multivariate version of Hensel's Lemma to reduce the problem to
solving binomial square systems over the residue field.

\bigskip

Our bound for the number of zeros of sparse multivariate square system of polynomials should
be compared with the bound obtained by J.M.~Rojas in 2004~\cite{Roj04}, which can be regarded
as the $p$-adic counterpart of A.~Khovanskii's theorem for fewnomials over the reals~\cite{Kho91},
or as the extension of Lenstra's estimates in the univariate case~\cite{Len99}.
Rojas showed that, over any finite extension $K/\Q_p$, any such system of polynomials has at most
$1+(C_Kn(t-n)^3\log(t-n))^n$ zeros, where $t$ is the total number of different exponents vectors
appearing in polynomials and $C_K$ is a computable constant that depends only on $K$. Our counting
gives a stronger bound, although only for regular systems:

\begin{thm}
Let $F=(f_1,\ldots,f_n)$ be a {\em regular}\footnote{see definition~\ref{def-reg}.} system of polynomials
in $K\vars$. Assume that the residue field $k$ is finite. Then the number of zeros of $F$ in
$(K^\ast)^n$ is at most $\binom{t_1}{2}\cdots\binom{t_n}{2}|k^\ast|^n$, where $t_i$ is the number of nonzero
monomials of $f_i$.
\end{thm}

This represents an improvement from roughly~$t^{3n}$ to~$t^{2n}$ in the case of regular systems.

\bigskip

Let $K$ be a complete field with respect to a discrete non-archimedian valuation $v:K\to\R\cup\{\infty\}$.
Let $A=\{x\in K\,:\,v(x)\geq 0\}$ be the valuation ring of~$K$. The ring~$A$ is local with maximal ideal
$\M=\{x\in K\,:\,v(x)>0\}$, which is principal $\M=\pi A$ since $v$ is discrete. We denote by~$k=A/\M$
the residue field of~$K$ with respect to~$v$. We denote the first digit of $x\in K^\ast$ by $\delta(x)=\pi^{-v(x)/v(\pi)}x\mod\M$.
The map $\delta:K^\ast\to k^\ast$ is a homomorphism, that can be seen as the composition of the homomorphisms
$$K^\ast\to\Z\times A^\ast\to A^\ast\to k^\ast,$$
where the first map is the isomorphism $x\mapsto (v(x)/v(\pi),\pi^{-v(x)/v(\pi)}x)$, the second arrow is the
projection on the second factor, and the third arrow is the reduction modulo $\M$.

\bigskip

Fix a set $\Delta\subseteq A\setminus\M$ of representatives of the first digit map. For any $x\in K^\ast$,
we write $\Delta(x)$ the representative corresponding to $\delta(x)$. Any element in $x\in K^\ast$ can be factorized
as $x=\pi^{v(x)/v(\pi)}\Delta(x)e(x)$ where $e(x)=x\pi^{-v(x)/v(\pi)}\Delta(x)^{-1}\in 1+\M$. Moreover, this
is the only possible factorization of $x$ as the product of a power of $\pi$, an element in $\Delta$, and an element in $1+\M$.
This implies that the map $K^\ast\to v(\pi)\Z\times k^\ast\times (1+\M)$ given by $x\mapsto(v(x),\delta(x),e(x))$ is
a bijection. The spirit behind most of our results is this bijection: we compute/count the solutions of systems of polynomials
by first looking at the valuation, then the first digit, and then the tail in $1+\M$. Our notions of genericity and
randomness are also based on the bijection.

\bigskip

Consider a square system $F=(f_1,\ldots,f_n)$ of $n$ polynomials in $K\vars$. Denote by $Z_K(F)$ the set of solutions
of $F$ in $(K^\ast)^n$. The study of the set $Z_K(F)$ that we do in this paper is based on the following program:
\begin{enumerate}
\item Study the set $S(F)=\{v(x)\,:\,x\in Z_K(F)\}\subseteq v(\pi)\Z^n$.
\item For each $w\in S(F)$ study the set $$D_w(F)=\{\delta(x)\,:\,x\in Z_K(F),\,v(x)=w\}\subseteq (k^\ast)^n.$$
\item For each $w\in S(F)$ and $\varepsilon\in D_w(F)$ study the set $$E_{w,\varepsilon}(F)=\{e(x)\,:\,x\in Z_K(F),\,v(x)=w,\,\delta(x)=\varepsilon\}
\subseteq(1+\M)^n.$$
\end{enumerate}
A similar program was successfully used by B.~Sturmfels and D.~Speyer in~\cite{SS04}, working on the field of Puisseax
series $\C\{\{t\}\}$, to give a simple proof of Kapranov's Theorem: item~1 correspond with their Theorem~2.1 and item~2
with Corollary~2.2.

\bigskip

Our approach for the first problem requires us to work only with the valuations of the coefficients and the exponent
vectors of the monomials of~$F$. We will prove that $S(F)\subseteq{\rm Trop}(F)\cap v(\pi)\Z^n$, where the set
${\rm Trop}(F)={\rm Trop}(f_1)\cap\cdots\cap{\rm Trop}(f_n)$ is the tropical prevariety induced by $F$. Recall that
for a given polynomial $f=\sum_{i=1}^ta_iX^{\alpha_i}\in K\vars$, the set ${\rm Trop}(f)$ is defined as the set of all possible
$w\in\R^n$ such that $v(a_i)+w\cdot\alpha_i$ for $i=1,\ldots,t$ reaches its minimum value at least twice. For any $w\in\R^n$,
the initial form ${\rm in}_w(f)\in k\vars$ is defined as the sum of $\delta(a_i)X^{\alpha_i}$, but including only the terms
that minimize $v(a_i)+w\cdot\alpha_i$. All the notions of tropical geometry used in this paper are defined in Section~\ref{sec-trop}
and can also be found in the literature in~\cite{SS04,Stur1,Stur3}.

\bigskip

For the second problem, we introduce the notion of $w$-semiregularity at a given
$w\in{\rm Trop}(F)\cap v(\pi)\Z^n$, that guarantees that $D_w(F)$ coincides with the set of zeros of the initial form
system ${\rm in}_w(F)$ in~$(k^\ast)^n$. In a few words, semiregularity at $w$ is a condition on $F$ that reformulates the
hypothesis of Hensel's Lemma (see~\cite[Pag.~48]{Rob00}) for zeros of valuation $w$ and for polynomials with coefficients in $K$ instead of $A$.
Semiregularity at $w$ also provides the solution of the third problem: for each $w\in{\rm Trop}(F)$ and $\varepsilon\in D_w(F)$, there is
exactly one solution of $F$ in $(K^\ast)^n$ with valuation vector $w$ and first digits $\varepsilon$, i.e. the set
$E_{w,\varepsilon}(F)$ has only one element. In  particular, for a $w$-semiregular system of polynomials $F$, where $w\in
{\rm Trop}(F)\cap v(\pi)\Z^n$, the first digit map $\delta:(K^\ast)^n\to(k^\ast)^n$ provides a bijection between roots of $F$
with valuation $w$ and roots of the initial form system ${\rm in}_w(F)$ in $(k^\ast)^n$. The definition of semiregularity (that
was obtained by keeping track several changes of variables carefully) and the main root counting theorem (proven by undoing all
these changes of variables) are presented in detail in Section~\ref{sec-semi} and summarized in the following statement:

\begin{thm}
Let $F=(f_1,\ldots,f_n)$ be a system of polynomial equations in $K\vars$. Let $w\in v(\pi)\Z^n$ be an isolated point of ${\rm Trop}(F)$.
If the initial form system ${\rm in}_w(F)$ has no degenerate zeros in $(k^\ast)^n$, then the first digit map induces a bijection
between the set of zeros of $F$ in $(K^\ast)^n$ with valuation $w$ and the set of zeros of ${\rm in}_w(F)$ in $(k^\ast)^n$.
\end{thm}

\bigskip

As a consequence of the results described in the last paragraph, we derive explicit formulas (or more precisely, an
algorithm) to compute efficiently the number of roots in $(K^\ast)^n$ of a large class of systems of polynomial equations.
These systems, called regular, are characterized by having a finite tropical prevariety, by being semiregular at any
point, and by having initial forms consisting only of binomials. Our notion of regularity and the formulas for the
number of roots generalize those shown in~\cite[Def.~1, Thm.~4.5]{AvIbr} to the multivariate case. All this work is
done in Section~\ref{sec-reg}.

\bigskip

Although regularity seems to impose a very strong
constraint on the system, we prove in Section~\ref{sec-gen} that this is not actually the case: regularity occurs generically when the residue
field $k$ has characteristic zero. The notion of genericity implicit in the previous statement (called tropical genericity)
refers to coefficients whose valuation vector do not lie in the union of certain hyperplanes. This notion is the natural
extension of the genericity in the algebraic geometry sense to tropical geometry.

\bigskip

Since we have explicit formulas for the number of roots of generic polynomials (with given support), we should be able to
compute the expected number of roots in $(K^\ast)^n$ of random polynomials. The only problem is that we need a way of
choosing the coefficients at random that produce tropically generic systems with probability $1$. Since our root counting
formula does not depend on the tail in $1+\M$ of the coefficients, we only need a way of selecting the valuation
of the coefficients and their first digits. The approach that we use consists of choosing the valuation at random uniformly
in an interval $[-M,M]$ and then letting $M$ go to infinity. The first digits are selected uniformly from $k^\ast$
when $k$ is a finite field, or in the case of $k=\R$ with any probability measure that gives equal probability to $\R_{>0}$ and $\R_{<0}$.
In the case that $k$ is algebraically closed, any selection of the first digits gives the same number of roots, and therefore
no probability measure in $k$ is needed.

\bigskip

Let $\mathcal{A}=\{\alpha_1<\alpha_2<\cdots<\alpha_t\}\subset\Z$ be a finite set ($t\geq 2$) and
consider an univariate polynomial $f\in K[X]$ with ${\rm supp}(f)=\mathcal{A}$ and random coefficients (chosen as explained above).
Let $E(\mathcal{A},K)$ be the limit of the expected number of roots of $f$ in $K^\ast$ as $M$ goes to infinity. Our main result of
section~\ref{sec-prob}, is a general formula for $E(\mathcal{A},K)$. As a particular case, we have the following result that it is
interesting in itself, and simple enough to be stated in this introduction:

\begin{thm}
Let $\mathcal{A}=\{\alpha_1<\alpha_2<\cdots<\alpha_t\}\subset\Z$ be a finite set with $t\geq 2$. If $k$ is algebraically closed
with ${\rm char}(k)=0$ or ${\rm char}(k)>\max_{\alpha,\beta\in\mathcal{A}}|\alpha-\beta|$, then
$$2-\frac{2}{t}\leq E(\mathcal{A},K)\leq 2\ln(t).$$
\end{thm}

A previous estimation for the expected number of roots of random polynomials with $p$-adic coefficients, although for a different
distribution (related to the Haar measure on $\Z_p$) was obtained by S.~Evans in~\cite{Evans06}.

\section{Tropical hypersurface induced by a Laurent polynomial}\label{sec-trop}

The main goal of this section is to introduce the reader the notions of tropical geometry
used in the rest of the paper.

\begin{defn}\label{def-trop}
Let $f\in K\vars$ be a polynomial with $t$ non-zero terms $f=\sum_{i=1}^ta_iX^{\alpha_i}$ where $a_i\in K^\ast$ and $\alpha_i=(\alpha_{i1},\ldots,\alpha_{in})\in\Z^n$ for all $i=1,\ldots,t$. We define the tropicalization of $f$ as the
piecewise linear function ${\rm tr}(f;w)=\min\{ l_i(f;w)\;i=1,\ldots,t\}$ where $l_i(f;w)=v(a_i)+\alpha_i\cdot w$.
The tropical hypersurface induced by $f$ is the set
$${\rm Trop}(f)=\{w_0\in\R^n\;:\;{\rm tr}(f;w)\;\mbox{is not differentiable at}\;w_0\}.$$
\end{defn}

The value of $l_i(f;w)$ is usually referred in the literature as the $w$-weight of the $i$-th term of $f$.

\begin{lemma}\label{trop1}
Let $f\in K\vars$ be a polynomial with $t$ terms and let $w_0\in\R^n$. Then $w_0\in {\rm Trop}(f)$ if and only if there are
indices $1\leq i<j\leq t$ such that $l_i(f;w_0)=l_j(f;w_0)\leq l_k(f;w_0)$ for all $k=1,\ldots,t$.
\end{lemma}
\begin{proof}
$(\Leftarrow)$ Assume first that $l_i(f;w_0)<l_k(f;w_0)$ for all $k\neq i$. Since the functions $l_i(f;w)$ are continuous, all these
inequalities remain valid in a neighborhood $U$ of $w_0$, and then ${\rm tr}(f;w)$ coincides with the linear function
$l_i(f;w)$ in $U$. In particular, ${\rm tr}(f;w)$ is differentiable at $w_0$, i.e. $w_0\not\in{\rm Trop}(f)$.

$(\Rightarrow)$ Now take $w_0\not\in{\rm Trop}(f)$. Since ${\rm tr}(f;w)$ is differentiable at $w_0$, then the linear function
$l(w)={\rm tr}(f;w_0)+\nabla{\rm tr}(f;w_0)\cdot(w-w_0)$ approximates ${\rm tr}(f;w)$ with order two near $w_0$, and since
${\rm tr}(f;w)$ is piecewise linear, then ${\rm tr}(f;w)=l(w)=l_i(f;w)$ for some $1\leq i\leq t$ in a neighborhood $U$ of $w_0$.
Therefore, for any other index $k\neq i$, we have that ${\rm tr}(f;w)=l_i(f;w)\leq l_k(f;w)$ in~$U$, or equivalently, $l_i(f;w_0)-l_k(f;w_0)\leq(\alpha_k-\alpha_i)\cdot(w-w_0)$ in~$U$. The right hand side of this inequality can be made strictly
negative by selecting $w-w_0$ a vector with the direction of $\alpha_i-\alpha_k$, hence $l_i(f;w_0)<l_k(f;w_0)$ for all $k\neq i$.
\end{proof}

Note that for any $x\in(K^\ast)^n$, the valuation of the $i$-th term of $f$ at $x$ is given by $l_i(f;v(x))$.

\begin{prop}\label{trop2}
Let $f\in K\vars$ and let $x\in(K^\ast)^n$ be a zero of~$f$. Then $v(x)\in{\rm Trop}(f)$.
\end{prop}
\begin{proof}
Sort all the $t$ monomials of $f$ according to their valuation at $x$.
$$l_{i_1}(f;v(x))\leq l_{i_2}(f;v(x))\leq\cdots\leq l_{i_t}(f;v(x))$$
Since the sum of all the monomials at $x$ is zero, the first two valuations in this list must coincide.
We conclude from Lemma~\ref{trop1} that $v(x)\in{\rm Trop}(f)$.
\end{proof}

\begin{defn}\label{defn-lp} 
Let $f\in K\vars$ be a polynomial with $t$ non-zero terms $f=\sum_{i=1}^ta_iX^{\alpha_i}$ and let $w\in\R^n$.
We define the lower polynomial $f^{[w]}$ of $f$ with respect to the valuation vector $w$ as
$$f^{[w]}=\!\!\!\!\!\!\!\!\!\!\sum_{i\;:\;l_i(f;w)={\rm tr}(f;w)}\!\!\!\!\!\!\!\!\!\!a_iX^{\alpha_i}\in K\vars.$$
We also define the initial form ${\rm in}_w(f)$ of $f$ with respect to $w$ as
$${\rm in}_w(f)=\!\!\!\!\!\!\!\!\!\!\sum_{i\;:\;l_i(f;w)={\rm tr}(f;w)}\!\!\!\!\!\!\!\!\!\!\delta(a_i)X^{\alpha_i}\in k\vars.$$
\end{defn}

Note that, according to Lemma~\ref{trop1}, $w\in{\rm Trop}(f)$ if and only if ${\rm in}_w(f)$ has at least two terms. This can
be taken as an alternative definition of the tropical hypersurface.
A key property of the initial forms is that if $x\in(K^\ast)^n$ is a solution of $f$ with $v(x)=w$, then $\delta(x)\in(k^\ast)^n$ is a
solution of ${\rm in}_{w}(f)$, as shown in the following lemma.

\begin{lemma}\label{tropz}
Let $f\in K\vars$, let $w\in\R^n$, let $x\in(K^\ast)^n$ with $v(x)=w$, and let $1\leq j\leq n$. Then:
\begin{enumerate}
\item $\pi^{-{\rm tr}(f;w)/v(\pi)}f(x)\in A$.
\item $\pi^{-{\rm tr}(f;w)/v(\pi)}f(x)\equiv{\rm in}_w(f)(\delta(x))\bmod{\M}$.
\item $\pi^{(w_j-{\rm tr}(f;w))/v(\pi)}\frac{\partial f}{\partial X_j}(x)\in A$.
\item $\pi^{(w_j-{\rm tr}(f;w))/v(\pi)}\frac{\partial f}{\partial X_j}(x)\equiv\frac{\partial {\rm in}_w(f)}{\partial X_j}(\delta(x))\bmod{\M}$.
\end{enumerate}
\end{lemma}
\begin{proof}
Let $f=\sum_{i=1}^ta_iX^{\alpha_i}\in K\vars$. The valuation of the $i$-th term of $f(x)$ is $l_i(f;w)$ and the minimum of all these valuations
is ${\rm tr}(f;w)$. This proves that $\pi^{-{\rm tr}(f;w)/v(\pi)}f(x)\in A$. Moreover, if $l_i(f;w)>{\rm tr}(f;w)$, then the $i$-term of $f(x)$
multiplied by $\pi^{-{\rm tr}(f;w)/v(\pi)}$ reduces to zero modulo $\M$, so $\pi^{-{\rm tr}(f;w)/v(\pi)}f(x)\equiv\pi^{-{\rm tr}(f;w)/v(\pi)}f^{[w]}(x)\bmod{\M}$.
Besides, all the terms in $\pi^{-{\rm tr}(f;w)/v(\pi)}f^{[w]}(x)$ have valuation zero, so reducing it modulo~$\M$ is
the same as adding the first digit of each term. This proves that $\pi^{-{\rm tr}(f;w)/v(\pi)}f(x)\equiv{\rm in}_w(f)(\delta(x))\bmod{\M}$.
The partial derivative of $f$ with respect to $X_j$ is $\partial f/\partial X_j=\sum_{i=1}^ta_i\alpha_{i,j}X^{\alpha_i-e_j}$, where
$\{e_1,\ldots,e_n\}$ is the standard basis of $\R^n$. The valuation of the $i$-th term of $\partial f/\partial X_j(x)$ is $l_i(f;w)-w_j+v(\alpha_{i,j})$,
thus $\pi^{(w_j-{\rm tr}(f;w))/v(\pi)}\partial f/\partial X_j(x)\in A$. Finally, in the reduction of $\pi^{(w_j-{\rm tr}(f;w))/v(\pi)}\partial f/\partial X_j(x)$ 
modulo $\M$, all the terms with $l_i(f;w)>{\rm tr}(f;w)-w_j$ dissapear, as well as the terms with $v(\alpha_{i,j})>0$. The remaining terms have all
valuation zero, and their first digits coincide with those of $\partial {\rm in}_w(f)/\partial X_j(\delta(x))$.
\end{proof}

The following lemma shows that the notions of tropicalization, tropical hypersurface, lower polynomial, and initial form, behave well
under rescaling of the variables and multiplication by monomials.

\begin{lemma}\label{trop4}
Let $f\in K\vars$, $a\in K^\ast$, $b=(b_1,\ldots,b_n)\in (K^\ast)^n$, $\alpha\in\Z^n$ and $w\in\R^n$.
\begin{enumerate}
\item ${\rm tr}(aX^\alpha f; w)={\rm tr}(f;w)+v(a)+\alpha\cdot w$.
\item ${\rm Trop}(aX^\alpha f)={\rm Trop}(f)$.
\item $(aX^\alpha f)^{[w]}=aX^\alpha f^{[w]}$.
\item ${\rm in}_w(aX^\alpha f)=\delta(a)X^\alpha {\rm in}_w(f)$.
\item ${\rm tr}(f(b_1X_1,\ldots,b_nX_n);w)={\rm tr}(f;w+v(b))$.
\item ${\rm Trop}(f(b_1X_1,\ldots,b_nX_n))={\rm Trop}(f)-(v(b_1),\ldots,v(b_n))$.
\item $f(b_1X_1,\ldots,b_nX_n)^{[w]}=f^{[w+v(b)]}(b_1X_1,\ldots,b_nX_n)$.
\item ${\rm in}_w(f(b_1X_1,\ldots,b_nX_n))={\rm in}_{w+v(b)}(f)(\delta(b_1)X_1,\ldots,\delta(b_n)X_n)$.
\end{enumerate}
\end{lemma}
\begin{proof}
Items~1 and~5 follow immediately from the identities $l_i(aX^\alpha f;w)=l_i(f;w)+v(a)+\alpha\cdot w$ and
$l_i(f(b_1X_1,\ldots,b_nX_n);w)=l_i(f;w+v(b))$. Items~2 and~6 are consequences of the previous
two and the definition of tropical hypersurface. The indices of the monomials of $f$ that are in
$(aX^\alpha f)^{[w]}$ correspond with the indices that minimize the value of $l_i(aX^\alpha f;w)$. Since
$v(a)+\alpha\cdot w$ is a constant, these indices also minimize $l_i(f;w)$, i.e. they correspond with the monomials
of $f$ in $f^{[w]}$. Therefore $(aX^\alpha f)^{[w]}=aX^\alpha f^{[w]}$. Similarly, the indices of
the terms of $f$ in $f(b_1X_1,\ldots,b_nX_n)^{[w]}$ minimize the expression $l_i(f(b_1X_1,\ldots,b_nX_n);w)$,
and therefore, coincide with the same indices of the monomials in $f^{[w+v(b)]}(b_1X_1,\ldots,b_nX_n)$. This proves
that $f(b_1X_1,\ldots,b_nX_n)^{[w]}=f^{[w+v(b)]}(b_1X_1,\ldots,b_nX_n)$. Finally, items~4 and~8 follow from~3 and~7
by taking the first digit of all the terms.
\end{proof}

In the next two lemmas, we show the relation between ${\rm Trop}(f)$ and ${\rm Trop}(f^{[w]})$ for any $w\in\R^n$.
It is clear that if $w\not\in{\rm Trop}(f)$, then $f^{[w]}$ is a single monomial, and therefore ${\rm Trop}(f^{[w]})=\emptyset$.
Otherwise, when $w\in{\rm Trop}(f)$, we have $w\in{\rm Trop}(f^{[w]})$ and ${\rm tr}(f;w)={\rm tr}(f^{[w]};w)$.
We will prove next that the tropical hypersurface ${\rm Trop}(f^{[w]})$ is a cone centered at $w$, that coincides with ${\rm Trop}(f)$
in a neighborhood of $w$. This completely characterizes ${\rm Trop}(f^{[w]})$ in terms of ${\rm Trop}(f)$.

\begin{lemma}\label{trop5}
Let $f\in K\vars$ and let $w\in{\rm Trop}(f)$. Then, for any $w'\in{\rm Trop}(f^{[w]})$, the ray $w+\lambda (w'-w)$ with $\lambda\geq 0$ is contained in ${\rm Trop}(f^{[w]})$.
\end{lemma}
\begin{proof}
Let $t$ be the number of terms of $f$. Write the lower polynomial of $f$ at $w$ as
 $f^{[w]}=a_{i_1}X^{\alpha_{i_1}}+\cdots+a_{i_r}X^{\alpha_{i_r}}$
where ${1\leq i_1<i_2<\cdots<i_r\leq t}$ are all the indices that minimize the linear functions $l_i(f;w)$.
The $s$-th term of $f^{[w]}$ is the $i_s$-th term of $f$. In particular, we have that
$l_s(f^{[w]};w)=l_{i_s}(f;w)={\rm tr}(f;w)$ for all $s=1,\ldots,r$. Since $w'\in{\rm Trop}(f^{[w]})$ we have, by
Lemma~\ref{trop1}, two indices $1\leq n<m\leq r$ such that $l_n(f^{[w]};w')=l_m(f^{[w]};w')\leq l_s(f^{[w]};w')$
for all $s=1,\ldots,r$.
Subtracting ${\rm tr}(f;w)$, multiplying by $\lambda\geq 0$ and then adding ${\rm tr}(f;w)$ to these (in)equalities we get
$$l_n(f^{[w]};w+\lambda(w'-w))=l_m(f^{[w]};w+\lambda(w'-w))\leq l_s(f^{[w]};w+\lambda(w'-w))$$
for all $s=1,\ldots,r$. This implies, by Lemma~\ref{trop1}, that $w+\lambda(w'-w)$  is in ${\rm Trop}(f^{[w]})$.
\end{proof}

\begin{lemma}\label{trop6}
Let $f\in K\vars$ and let $w\in{\rm Trop}(f)$. Then there exists $\varepsilon>0$ such that
${\rm Trop}(f)\cap B_\varepsilon(w)={\rm Trop}(f^{[w]})\cap B_\varepsilon(w)$.
\end{lemma}
\begin{proof}
Let $t$ be the number of terms of $f$. Let $I=\{1\leq i\leq t\;:\; l_i(f;w)={\rm tr}(f;w)\}$ be the set of indices of the
monomials of $f$ in $f^{[w]}$. Note that $l_i(f;w)<l_k(f;w)$ for all $i\in I$ and $k\not\in I$. Since $l_i(f;\cdot):\R^n\to\R$
are continuous functions, there exists $\varepsilon>0$ such that
\begin{equation}\label{eq1}
l_i(f;w')<l_k(f;w')\qquad\forall\,w'\in B_\varepsilon(w),\;\forall\,i\in I,\;\forall\,k\not\in I.
\end{equation}
Take $w'\in{\rm Trop}(f)\cap B_\varepsilon(w)$. By Lemma~\ref{trop1}, there are indices ${1\leq i<j\leq t}$ such that
$l_i(f;w')=l_j(f;w')\leq l_k(f;w')$ for all $k=1,\ldots,t$. By the inequalities (\ref{eq1}), we conclude that $i,j\in I$.
Therefore, by Lemma~\ref{trop1}, $w'\in{\rm Trop}(f^{[w]})$.

Now take $w'\in{\rm Trop}(f^{[w]})\cap B_\varepsilon(w)$. By Lemma~\ref{trop1} we have two different indices $i,j\in I$ such that
$l_i(f;w')=l_j(f;w')\leq l_k(f;w')$ for all ${k\in I}$. By~(\ref{eq1}), this inequality holds also for $k\not\in I$.
This means, by Lemma~\ref{trop1}, that $w'\in{\rm Trop}(f)$.
\end{proof}

Lemma~\ref{trop1} gives a simple procedure to compute tropical hypersurfaces that requiere to solve systems
of linear equations and inequalities. The following is a simple geometric interpretation of that using
polyhedra.

\begin{defn}\label{def-npoly}
Let $f=\sum_{i=1}^ta_iX^{\alpha_i}\in K\vars$. The Newton Polytope of $f$, denoted ${\rm NP}(f)$, is the convex hull
of the set $$\{(\alpha_i,v(a_i))\,:\,i=1,\ldots,t\}\subseteq\R^{n+1}.$$
A hyperplane $H\subseteq\R^{n+1}$, not parallel to the line $x_1=\cdots=x_n=0$, is a supporting hyperplane of the Newton
Polytope of $f$ if ${\rm NP}(f)$ is included in the upper half-space\footnote{Up and down is understood with
respect to the variable $x_{n+1}$. The upper half-space of $H$ is well-defined since $H$ is not parallel to the vertical axis.}
determined by $H$ and ${\rm NP}(f)\cap H\neq\emptyset$.
\end{defn}

\begin{lemma}\label{trop7}
Let $f\in K\vars$. Then ${\rm Trop}(f)$ is the set of all ${w\in\R^n}$ such that $(w,1)\in\R^{n+1}$ is the
normal vector of a supporting hyperplane $H$ of ${\rm NP}(f)$ with $|H\cap{\rm NP}(f)|>1$.
\end{lemma}
\begin{proof}
Write $f=\sum_{i=1}^ta_iX^{\alpha_i}$. $(\subseteq)$ Take $w\in{\rm Trop}(f)$. By Lemma~\ref{trop1}, there are
two indices $1\leq i<j\leq t$ such that $l_i(f;w)=l_j(f;w)\leq l_k(f;w)$ for all $k=1,\ldots,t$.
This is equivalent to say that the hyperplane $$H=\{x\in\R^{n+1}\,:\,(w,1)\cdot x={\rm tr}(f;w)\},$$
with normal vector $(w,1)$, contains the points $(\alpha_i,v(a_i))$ and $(\alpha_j,v(a_j))$, and the upper
half-space $H^+$ determined by $H$ contains all the points $(\alpha_k,v(a_k))$. Since $H^+$ is convex, then
${\rm NP}(f)\subseteq H^+$. $(\supseteq)$ Now assume that $H$ is a supporting hyperplane with normal vector
$(w,1)$ that contains at least two points of the Newton Polytope of $f$. Since ${\rm NP}(f)$ is a polyhedron,
then $H$ contains at least two vertices $(\alpha_i,v(a_i))$ and $(\alpha_j,v(a_j))$. The remaining vertices
are contained in the upper half-space determined by $H$. This means that $\alpha_i\cdot w+ v(a_i)=
\alpha_j\cdot w+v(a_j)\leq \alpha_k\cdot w+v(a_k)$ for all $k=1,\ldots,t$, and by Lemma~\ref{trop1}, that
$w\in{\rm Trop}(f)$.
\end{proof}

In the case of an univariate polynomial $f\in K[X]$, Lemma~\ref{trop7} says that ${\rm Trop}(f)$ is
the set of minus the slope of the segments of the lower hull of ${\rm NP}(f)$.

\section{Semiregular systems of polynomial equations.}\label{sec-semi}

\begin{defn}\label{def-sys}
Consider a system $F$ of $n$ equations in $n$ variables.
$$F=\left\{\begin{array}{c}f_1(X_1,\ldots,X_n)=0 \\ \vdots \\ f_n(X_1,\ldots,X_n)=0\end{array}\right.$$
The equations are given by non-zero polynomials in $K\vars$ and the unknowns are in $K^{\ast}$. The system $F$
will be written $(f_1,\ldots,f_n)$ in order to simplify the notation.
We define the tropical prevariety ${\rm Trop}(F)$ induced by $F$ as
$${\rm Trop}(F)={\rm Trop}(f_1)\cap\cdots\cap {\rm Trop}(f_n).$$
For any $w\in {\rm Trop}(F)$ we denote by $F^{[w]}$ and ${\rm in}_w(F)$ the systems of polynomial equations given by
the lower polynomials $f_1^{[w]},\ldots,f_n^{[w]}$ and the initial forms ${\rm in}_w(f_1),\ldots,{\rm in}_w(f_n)$
respectively.
\end{defn}

By Proposition~\ref{trop2}, any solution $x\in(K^{\ast})^n$ of $F$ satisfies $v(x)\in{\rm Trop}(F)$.

\begin{lemma}\label{semi1}
Let $F$ be a system of $n$ polynomials in $K\vars$.
If~$w$ is an isolated point of ${\rm Trop}(F)$, then ${\rm Trop}(F^{[w]})=\{w\}$ and all the solutions
$x\in(K^\ast)^n$ of $F^{[w]}$ have valuation vector $v(x)=w$.
\end{lemma}
\begin{proof}
By Lemma~\ref{trop6}, the tropical prevarieties ${\rm Trop}(F)$ and ${\rm Trop}(F^{[w]})$ coincide in a neighborhood of $w$.
In particular, there exists $\varepsilon>0$ such that ${\rm Trop}(F^{[w]})\cap B_\varepsilon(w)=\{w\}$. On the other
hand, by Lemma~\ref{trop5}, the tropical prevariety ${\rm Trop}(F^{[w]})$ is a cone centered at $w$. This implies that
${\rm Trop}(F^{[w]})=\{w\}$. Therefore, by Proposition~\ref{trop2}, all the solutions $x\in(K^\ast)^n$ of $F^{[w]}$
have valuation vector $v(x)=w$.
\end{proof}

\begin{defn}\label{def-sreg}
Consider a system $F=(f_1,\ldots,f_n)$ of $n$ polynomials in $K\vars$, and let $w\in\R^n$.
We say that $F$ is semiregular at $w$ if either ${w\not\in{\rm Trop}(F)\cap v(\pi)\Z^n}$ or
${\rm in}_w(F)$ has no degenerate zero in $(k^\ast)^n$.
We say that $F$ is normalized at $w$ if ${\rm tr}(f_1;w)=\cdots={\rm tr}(f_n;w)=0$.
\end{defn}

\begin{lemma}\label{semiy}
Let $F=(f_1,\ldots,f_n)$ be a system of $n$ polynomials in $K\vars$ semiregular at $w\in\R^n$.
Then, for each zero $x\in(K^\ast)^n$ of $F$ with $v(x)=w$, we have $$v({\rm Jac}(F)(x))={\rm tr}(f_1;w)+\cdots+{\rm tr}(f_n;w)-(w_1+\cdots+w_n).$$
\end{lemma}
\begin{proof}
In the case $w\not\in{\rm Trop}(F)\cap v(\pi)\Z^n$, there are no zeros of $F$ with valuation $w$, and there is nothing to prove. Therefore,
we can assume without loss of generality that $w\in{\rm Trop}(F)\cap v(\pi)\Z^n$. Take a zero $x\in(K^\ast)^n$ of $F$ with valuation $v(x)=w$.
By Lemma~\ref{tropz}, the point $\delta(x)\in(k^\ast)^n$ is a zero of ${\rm im}_w(f)$, and then, by the semiregularity of $F$ at $w$, we
have ${\left.\det\left(\frac{\partial{\rm in}_w(f_i)}{\partial X_j}\right)\right|_{\delta(x)}\neq 0}$. Again by Lemma~\ref{tropz}, this means that
$\det\left(\pi^{w_j-{\rm tr}(f_i)}\left.\frac{\partial f_i}{\partial X_j}\right|_x\right)\not\equiv 0\bmod{\M}$, and by factoring
out the powers of $\pi$ of the determinant, we conclude that $v({\rm Jac}(F)(x))={\rm tr}(f_1;w)+\cdots+{\rm tr}(f_n;w)-(w_1+\cdots+w_n)$.
\end{proof}

The following three lemmas show how semiregularity behaves with respect to a rescaling of variables and multiplication by monomials.

\begin{lemma}\label{semi2}
Let $F=(f_1,\ldots,f_n)$ be a system of $n$ polynomials in $K\vars$.
Let $w\in\R^n$, $a_1,\ldots,a_n\in K^\ast$, and $\alpha_1,\ldots,\alpha_n\in\Z^n$. Then $F$ is semiregular
at $w$ if and only if $\tilde{F}=(a_1X^{\alpha_1}f_1,\ldots,a_nX^{\alpha_n}f_n)$ is semiregular at $w$.
\end{lemma}
\begin{proof}
By the item 2 of Lemma~\ref{trop4}, we have that ${\rm Trop}(F)={\rm Trop}(\tilde{F})$, and since the claim is symmetric,
it is enough to prove that when $w\in{\rm Trop}(F)\cap v(\pi)\Z^n$ and ${\rm in}_w(F)$ has no degenerate zero in $(k^\ast)^n$
then also ${\rm in}_w(\tilde{F})$ has no degenerate zero. By the item 4 of Lemma~\ref{trop4}, we have that
${\rm in}_w(\tilde{F})=(\delta(a_1)X^{\alpha_1}{\rm in}_w(f_1),\ldots,\delta(a_n)X^{\alpha_n}{\rm in}_w(f_n))$, and
in particular, ${\rm in}_w(F)$ and ${\rm in}_w(\tilde{F})$ have the same zeros in $(k^\ast)$. Let $x\in(k^\ast)^n$ be
one of these zeros, which by assumption is a non-degenerate zero of ${\rm in}_w(F)$. We have to show that $x$ is also
a non-degenerate zero of ${\rm in}_w(\tilde{F})$. The Jacobian of ${\rm in}_w(\tilde{F})$ is given by the following expression.
$${\rm Jac}({\rm in}_w(\tilde{F}))=\det\left(\delta(a_i)\alpha_{ij}X^{\alpha_i-e_j}{\rm in}_w(f_i)+\delta(a_i)X^{\alpha_i}\frac{\partial{\rm in}_w(f_i)}{\partial X_j}\right)_{1\leq i,j\leq n}$$
Evaluating at $X=x$ we get ${\rm Jac}({\rm in}_w(\tilde{F}))(x)=\delta(a_1\cdots a_n)x^{\alpha_1+\ldots+\alpha_n}{\rm Jac}({\rm in}_w(F))(x)$, which is
not zero in $k^\ast$, since $x$ is a non-degenerate zero of ${\rm in}_w(F)$.
\end{proof}

\begin{lemma}\label{semi3}
Let $F=(f_1,\ldots,f_n)$ be a system of $n$ polynomials in $K\vars$.
Let $w\in\R^n$ and $b=(b_1,\ldots,b_n)\in(K^\ast)^n$. Then $F$ is semiregular at $w$ if and only if the system
with rescaled variables $\tilde{F}=(f_1(b_1X_1,\ldots,b_nX_n),\ldots,f_n(b_1X_1,\ldots,b_nX_n))$ is semiregular at $w-v(b)$.
\end{lemma}
\begin{proof}
By the item 6 of Lemma~\ref{trop4}, we have that ${\rm Trop}(\tilde{F})={\rm Trop}(F)-v(b)$. Since $v(b)\in v(\pi)\Z^n$, then
$w\in{\rm Trop}(F)\cap v(\pi)\Z^n$ if and only if $w\in{\rm Trop}(\tilde{F})\cap v(\pi)\Z^n$. By the symmetry of the claim,
it is enough to show that when $w\in{\rm Trop}(F)\cap v(\pi)\Z^n$ and ${\rm in}_w(F)$ has no degenerate zero in $(k^\ast)^n$, then
also ${\rm in}_{w-v(b)}(\tilde{F})$ has no degenerate zero. By the item 8 of Lemma~\ref{trop4}, we have that
${\rm in}_{w-v(b)}(\tilde{F})=({\rm in}_w(f_1)(\delta(b)X),\ldots,{\rm in}_w(f_n)(\delta(b)X))$, and in particular,
if $x\in(k^\ast)^n$ is a zero of ${\rm in}_{w-v(b)}(\tilde{F})$, then $y=\delta(b)x$ is a zero of ${\rm in}_w(F)$. A simple
computation using the chain rule shows that ${\rm Jac}({\rm in}_{w-v(b)}(\tilde{F}))(x)=\delta(b_1\cdots b_n){\rm Jac}({\rm in}_w(F))(y)$. Since
the right hand side does not vanish at any zero $y$ of ${\rm in}_w(F)$, then the zeros of ${\rm in}_{w-v(b)}(\tilde{F})$ are
all non-degenerate.
\end{proof}

\begin{lemma}\label{semi5}
Let $F$ be a system of $n$ polynomials in $K\vars$ and let $w\in{\rm Trop}(F)$. Then $F$ is semiregular
(resp. normalized) at $w$ if and only if $F^{[w]}$ is semiregular (resp. normalized) at~$w$.
\end{lemma}
\begin{proof}
The claim that $F$ is normalized at $w$ if and only $F^{[w]}$ is normalized at $w$ follows from the fact
that ${\rm tr}(f_i;w)={\rm tr}(f_i^{[w]};w)$ for all $i=1,\ldots,n$.
The claim about semiregularity is immediate from ${\rm in}_w(F)={\rm in}_w(F^{[w]})$.
\end{proof}

At this point we have all the necessary ingredients for the main result of this section, which is a reformulation
of Hensel's Lemma in the language of Definition~\ref{def-sreg}. For pedagogical reasons, we start with the classical
statement, and then, we reformulate it progressively until we arrive to the final version in Corollary~\ref{semi8}.

\begin{lemma}[Hensel]\label{hensel}
Let $F$ be a system of $n$ polynomials in $A\vars$ and denote by $\overline{F}$ the system reduced modulo $\M$.
Let $\overline{x}\in(k^\ast)^n$ be a solution of $\overline{F}$ such that ${\rm Jac}(\overline{F})(\overline{x})\neq 0$.
Then there exists a unique solution $x\in (A\setminus\M)^n$ of $F$ such that $\overline{x}=x\mod\M$.
\end{lemma}
\begin{proof}
See~\cite[Prop.~2.11]{AKP06}.
\end{proof}

\begin{lemma}\label{semi4}
Let $F$ be a system of $n$ polynomials in $K\vars$ such that $0\in{\rm Trop}(F)$. Assume also that $F$ is normalized and
semiregular at $0$. Then all the coefficients of $F$ are in the valuation ring $A$. Moreover, the reduction
map$\mod\M:A^n\to k^n$ induces a bijection between the set of zeros of $F$ in $(K^\ast)^n$ with valuation vector~$0$
(i.e. in $(A\setminus\M)^n$) and the set of zeros of $\overline{F}$ in $(k^\ast)^n$.
\end{lemma}
\begin{proof}
Suppose that $F=(f_1,\ldots,f_n)$.
Since the system is normalized at $0$, we have ${\rm tr}(f_i;0)=0$ for all $i=1,\ldots,n$. Since ${\rm tr}(f_i;0)$ is the minimum
valuation of the coefficients of~$f_i$, then all the coefficents of~$f_i$ have valuation at least $0$, i.e. $f_i\in A\vars$.
Moreover, the terms of $f_i$ that are kept in ${\rm in}_0(f_i)$ are those with coefficients in $A\setminus\M$. For these
terms, reducing modulo~$\M$ or taking first digit is exaclty the same, so $\overline{f_i}={\rm in}_0(f_i)$. In particular, we
have that $\overline{F}={\rm in}_0(F)$ has no degenerate solutions in $(k^\ast)^n$.
It is clear that the reduction modulo~$\M$ maps zeros of $F$ in $(K^\ast)^n$ with valuation $0$ to zeros of $\overline{F}$
in $(k^\ast)^n$. We only have to show that the map is a bijection. For the surjectivity, take a zero of $\overline{F}$ in $(k^\ast)^n$.
The semiregularity at $0$ guarantees that it is non-degenerate zero, and Lemma~\ref{hensel} shows that it can be lifted to a
zero of $F$ in $(A\setminus\M)^n$, i.e. to a zero of $F$ with valuation $0$. The injectivity follows from the uniqueness of the
lifting in Hensel's Lemma.
\end{proof}

\begin{defn}\label{def-zeroset}
For any system of polynomials $F$ in $K\vars$, the set of roots of $F$ in $(K^\ast)^n$ is denoted by $Z_K(F)$, and the set
of zeros of $F$ with valuation $w$ is written $Z_K^w(F)$.
\end{defn}

\begin{thm}\label{semi6}
Let $F$ be a system of $n$ polynomials in $K\vars$. Let $w\in{\rm Trop}(F)\cap v(\pi)\Z^n$ and suppose that $F$ is semiregular at $w$.
The first digit maps $\delta:Z_K^w(F)\to Z_{k}({\rm in}_w(F))$ and $\delta:Z_K^w(F^{[w]})\to Z_{k}({\rm in}_w(F))$ are bijections
(and are well-defined between these sets of roots).
\end{thm}
\begin{proof}
The case $w=0$ and $F$ normalized at $0$ follows immediately from Lemmas~\ref{semi4} and~\ref{semi5} and the fact that the reductions of~$F$
and~$F^{[0]}$ modulo~$\M$ coincide with~${\rm in}_0(F)$. Note that the assumption that $F$ is normalized at $0$ can be easily removed by
pre-multiplying each equation in $F$ by a suitable constant in~$K^\ast$. We can also reduce the general case to $w=0$ by a simple change of
variables. Define $\hat{F}=F(\pi^{w_1/v(\pi)}X_1,\ldots,\pi^{w_n/v(\pi)}X_n)$. By Lemma~\ref{semi3}, the system $\hat{F}$ is semiregular at~$0$.
It is clear that the first digit preserving map $(x_1,\ldots,x_n)\mapsto(\pi^{w_1/v(\pi)}x_1,\ldots,\pi^{w_n/v(\pi)}x_n)$ is a bijection between
the set of solutions of $\hat{F}$ with valuation vector $0$ and the zeros of $F$ with valuation $w$. Moreover, by the item 8 of
Lemma~\ref{trop4}, we have ${\rm in}_w(F)={\rm in}_0(\hat{F})$, and by the item 7 we have
$F^{[w]}(\pi^{w_1/v(\pi)}X_1,\ldots,\pi^{w_n/v(\pi)}X_n)=\hat{F}^{[0]}$. This provides the reduction to the case $w=0$.
\end{proof}

Although the previous result contains all the substance of this section, the following corollary is the way Theorem~\ref{semi6}
is intended to be used in practice.

\begin{coro}\label{semi8}
Let $F$ be a system of $n$ polynomials in $K\vars$. Assume that $F$ is semiregular at $w$. Then there is a unique bijection
between the sets $Z_K^w(F)$ and $Z_K^w(F^{[w]})$ that preserves first digits. If $w\not\in{\rm Trop}(F)$ or $w\not\in v(\pi)\Z^n$,
then these sets are empty. Otherwise, the first digit map gives bijections from $Z_K^w(F)$ and $Z_K^w(F^{[w]})$ to
$Z_{k}({\rm in}_w(F))$.
\end{coro}

A more computational point of view is shown in the following algorithm.

\begin{algorithm}[H]
\caption{Decide whether a system $F=(f_1,\ldots,f_n)$ of $n$ polynomials in $K\vars$ is semiregular at a
given point $w=(w_1,\ldots,w_n)\in\R^n$. In case of semiregularity, print the number of solutions in $(K^\ast)^n$
with valuation vector $w$.}
\label{alg-semireg-test}
\begin{algorithmic}[1]
\IF {$w\not\in v(\pi)\Z^n$}
\PRINT the system has no solutions in $(K^\ast)^n$ with valuation $w$
\RETURN YES
\ENDIF
\FOR {$i=1,\ldots,n$}
\STATE $\tilde{f_i}\leftarrow {\rm in}_w(f_i)$
\IF {$\tilde{f_i}$ is a monomial}
\PRINT the system has no solutions in $(K^\ast)^n$ with valuation $w$
\RETURN YES
\ENDIF
\ENDFOR
\STATE ${\rm Jac}(\tilde{F})\leftarrow\det(\partial \tilde{f_i}/\partial X_j)$
\IF {there is a solution of $\tilde{f_1}(x)=\cdots=\tilde{f_n}(x)={\rm Jac}(\tilde{F})(x)=0$ in $(k^\ast)^n$}
\RETURN NO
\ENDIF
\STATE $s\leftarrow$ number of solutions of $\tilde{f_1}(x)=\cdots=\tilde{f_n}(x)=0$ in $(k^\ast)^n$
\PRINT the system has $s$ solutions in $(K^\ast)^n$ with valuation $w$
\RETURN YES
\end{algorithmic}
\end{algorithm}

In case that only an estimation for the number of zeros is needed, the following statement might be useful.

\begin{coro}\label{semi7}
Let $F$ be a system of $n$ polynomials in $K\vars$.
If ${\rm Trop}(F)$ is finite and $F$ is semiregular at any $w\in{\rm Trop}(F)$, then the number of solutions of $F$ 
in $(K^\ast)^n$ is $$|Z_K(F)|=\!\!\!\!\!\!\!\!\!\!\!\!\!\!\sum_{w\in{\rm Trop}(F)\cap v(\pi)\Z^n}\!\!\!\!\!\!\!\!\!\!\!\!\!\!|Z_k({\rm in}_w(F)|
\leq |{\rm Trop}(F)\cap v(\pi)\Z^n|\cdot|k^\ast|^n\leq |{\rm Trop}(F)|\cdot|k^\ast|^n.$$
\end{coro}

Note that when ${\rm Trop}(F)$ is a finite set, then it has at most $\prod_{i=1}^n\binom{t_i}{2}$ points, where $t_i$ is the number of monomials
of $f_i$. Each ${\rm Trop}(f_i)$ is contained in the union of $\binom{t_i}{2}$ hyperplanes (see Lemma~\ref{trop1}), and the intersection of $n$ of
these hyperplanes (one in each ${\rm Trop}(f_i)$) determines at most one point in ${\rm Trop}(F)$. In particular, a system $F$ that satisfies
the hypothesis of Corollary~\ref{semi7} has at most $\binom{t_1}{2}\cdots\binom{t_n}{2}|k^\ast|^n$ roots in $(K^\ast)^n$, and all these roots are
non-degenerate.

\bigskip

We conclude this section with a discussion of the univariate case. Consider $f=\sum_{i=1}^ta_iX^{\alpha_i}\in K[X]$. In section~\ref{sec-trop},
we showed that the tropical hypersurface of $f$ is the set of minus the slope of the segments of the lower hull of ${\rm NP}(f)$. For each
of these $w\in{\rm Trop}(f)$, the lower polynomial $f^{[w]}$ and initial form ${\rm in}_w(f)$ are simply the polynomials obtained by keeping
only the terms with $(\alpha_i,v(a_i))$ lying on the segment of slope $-w$. For each $w\in{\rm Trop}(f)$, semiregularity at $w$ means that
either $w\not\in v(\pi)\Z$, in which case $f$ has no solutions in $K^\ast$ with valuation $w$, or ${\rm in}_w(f)$ has no degenerate zeros in
$k^\ast$. In case of semiregularity at $w\in{\rm Trop}(f)\cap v(\pi)\Z$, our main result says that the number of roots of $f$ in $K^\ast$ with
valuation $w$ and the number of roots of ${\rm in}_w(f)$ in $k^\ast$ coincide.

\section{Regularity.}\label{sec-reg}

\begin{defn}\label{def-reg}
A system $F$ of $n$ polynomials in $K\vars$ is regular if ${\rm Trop}(F)$ is finite, $F^{[w]}$ consists solely of binomials
and $F$ is semiregular at~$w$ for all $w\in{\rm Trop}(F)$.
\end{defn}

For this kind of system, we can provide an explicit formula for the number of roots in $(K^\ast)^n$. We will also give
a different characterization of regularity that is easier to check. First of all, the notion of regularity is well-behaved
under monomial changes of variables.

\begin{lemma}\label{reg1}
Let $F=(f_1,\ldots,f_n)$ be a system of polynomials in $K\vars$.
Let $a_1,\ldots,a_n\in K^\ast$, $b_1,\ldots,b_n\in K^\ast$, and $\alpha_1,\ldots,\alpha_n\in\Z^n$.
The following three statements are equivalent.
\begin{enumerate}
\item $F$ is regular.
\item $(a_1X^{\alpha_1}f_1,\ldots,a_nX^{\alpha_n}f_n)$ is regular.
\item $(f_1(b_1X_1,\ldots,b_nX_n),\ldots,f_n(b_1X_1,\ldots,b_nX_n))$ is regular.
\end{enumerate}
\end{lemma}
\begin{proof}
A consequence of Lemmas~\ref{trop4}, \ref{semi2} and \ref{semi3}.
\end{proof}

The problem of deciding whether a system is regular or not can be reduced to the case of binomial systems: in Definition~\ref{def-reg}, the
condition $F$ is semiregular at $w$ can be replaced, according to Lemma~\ref{semi5}, by the condition $F^{[w]}$ is semiregular at $w$. The
following lemma and proposition characterize semiregularity for binomial systems.

\begin{lemma}\label{reg2}
Consider a binomial system $B=(a_1X^{\alpha_1}-b_1X^{\beta_1},\ldots,a_nX^{\alpha_n}-b_nX^{\beta_n})$ with coefficients
$a=(a_1,\ldots,a_n)\in (K^\ast)^n$, let $b=(b_1,\ldots,b_n)\in(K^\ast)^n$, and let $M\in\Z^{n\times n}$ be the matrix whose $i$-th
row is $\alpha_i-\beta_i$ for $i=1,\ldots,n$. Then
$${\rm Trop}(B)=\{w\in\R^n\;:\; Mw=v(b)-v(a)\}.$$
In particular, ${\rm Trop}(B)$ is finite (and non-empty) if and only if $det(M)\neq 0$.
\end{lemma}
\begin{proof}
By Lemma~\ref{trop1}, the tropical hypersurface of the $i$-th binomial is
${\rm Trop}(a_iX^{\alpha_i}-b_iX^{\beta_i})=\{w\in\R^n\;:\;v(a_i)+\alpha_i\cdot w = v(b_i)+\beta_i\cdot w\}$.
This equation corresponds with the $i$-th row of $Mw=v(b)-v(a)$.
\end{proof}

For any vector $x=(x_1,\ldots,x_n)$ with non-zero entries and any matrix $M=(m_{ij})_{1\leq i,j\leq n}\in\Z^{n\times n}$, we write
$$x^M=(x_1^{m_{11}}\cdots x_n^{m_{1n}},\ldots,x_1^{m_{n1}}\cdots x_n^{m_{nn}}).$$
Note that if $P,Q\in\Z^{n\times n}$, then $x^{PQ}=(x^Q)^P$.

\begin{prop}\label{reg3}
Consider the binomial system $B=(a_1X^{\alpha_1}-b_1X^{\beta_1},\ldots,a_nX^{\alpha_n}-b_nX^{\beta_n})$ in $K\vars$.
Let $a=(a_1,\ldots,a_n)$ and $b=(b_1,\ldots,b_n)$. Assume that the matrix $M\in\Z^{n\times n}$, whose $i$-th row is $\alpha_i-\beta_i$
for $i=1,\ldots,n$, has non-zero determinant. Let $M=PDQ$
be the Smith Normal Form of $M$, i.e. $P,Q\in\Z^{n\times n}$ are invertible and $D={\rm diag}(d_1,\ldots,d_n)$ with
$d_1\,|\,d_2\,|\cdots|\,d_n$ positive integers. Then $B$ is semiregular at $w=M^{-1}(v(b)-v(a))$ if and only if either:
\begin{enumerate}
\item $w\not\in v(\pi)\Z^n$.
\item ${\rm char}(k)\nmid\det(M)$.
\item the $i$-th entry of $(\delta(b_1/a_1),\ldots,\delta(b_n/a_n))^{P^{-1}}$ is not a $d_i$-th power in $k^\ast$ for some $i=1,\ldots,n$.
\end{enumerate}
In this case, if $(1)$ and $(3)$ do not hold, then the number of solutions of the system $B$ in $(K^\ast)^n$ is
$|Z_K(B)|=\prod_{i=1}^n|\{\xi\in k^\ast\;:\:\xi^{d_i}=1\}|$. Otherwise $B$ has no solutions in $(K^\ast)^n$.
\end{prop}
\begin{proof}
By Lemma~\ref{reg1}, we have $w\in{\rm Trop}(B)$. In case that $w\not\in v(\pi)\Z^n$, then $B$ is semiregular at~$w$ by definition,~$B$
has no solutions in $(K^\ast)^n$ since there are no elements in $(K^\ast)^n$ with valuation $w$, and the proposition is proven. Now assume
that $w\in v(\pi)\Z^n$. By Lemma~\ref{semi2}, the system~$B$ is semiregular at~$w$ if and only if the system $X^M=b/a$ is semiregular at~$w$.
The initial form system is $X^M=\delta(b/a)$. Any solution $x\in(K^\ast)^n$ of this system satisfies $(x^Q)^D=(\delta(b/a))^{P^{-1}}$ and then
the condition of item~3 is not met. In other words, if the system satisfies the third condition, then the initial form system (and also $B$)
has no solution, $B$ is automatically semiregular at $w$, and the proposition in proven. So we can assume without loss of generality
that $B$ does not satisfy items~1 and~3. In this case, there exist $y\in(k^\ast)^n$ such that $y^D=(\delta(b/a))^{P^{-1}}$, and then
$x=y^{Q^{-1}}\in(k^\ast)^n$ is a zero of $X^M=\delta(b/a)$. The Jacobian of this system is $J=\det([m_{ij}X_1^{m_{i1}}\cdots X_j^{m_{ij}-1}\cdots X_n^{m_{in}}]_{1\leq i,j\leq n})$, which, after factoring out $X_j^{-1}$ from the $j$-th column, and then $X_1^{m_{i1}}\cdots X_n^{m_{in}}$ from
the $i$-th row, becomes a single term with coefficient $\det(M)$. In particular, a solution $x\in(k^\ast)^n$ of $X^M=\delta(b/a)$ is non-degenerate
if and only if ${\rm char}(k)\nmid\det(M)$. This shows the equivalence between semiregularity of $B$ at $w$ and item~2. 
Finally, the number of solutions of $X^M=\delta(b/a)$ is equal to the number of solutions of $Y^D=(\delta(b/a))^{P^{-1}}$, since the map $x\mapsto x^Q$
is a bijection. We know already that there is a solution $y\in(k^\ast)^n$, and it is clear that all other solution can be obtained by multiplying
the $i$-th entry of~$y$ by a $d_i$-th root of unity in $k^\ast$. This proves the formula for the number of zeros of $B$.
\end{proof}

A system of polynomials $F$ is regular if and only if ${\rm Trop}(F)$ is finite and $F^{[w]}$ is a binomial system that satisfies the
assumptions of Proposition~\ref{reg3} for all $w\in{\rm Trop}(F)$. In this case, an explicit formula for the number of roots of $F$ in
$(K^\ast)^n$ can be obtained from Corollary~\ref{semi7} and Proposition~\ref{reg3}. The following algorithm summarizes this procedure.

\begin{algorithm}[H]
\caption{Decides whether a system $F=(f_1,\ldots,f_n)$ of $n$ polynomials in $K\vars$ is regular. In case of regularity, it
prints the number of solutions in $(K^\ast)^n$.}
\label{alg-reg-test}
\begin{algorithmic}[1]
\STATE compute ${\rm Trop}(F)\leftarrow{\rm Trop}(f_1)\cap\cdots\cap{\rm Trop}(f_n)$
\IF {$|{\rm Trop}(F)|=\infty$}
\RETURN NO
\ENDIF
\STATE $s\leftarrow 0$
\FORALL {$w\in{\rm Trop}(F)$}
\IF {$f_1^{[w]},\ldots,f_n^{[w]}$ are not all binomials}
\RETURN NO
\ELSE
\STATE write each $f_i^{[w]}$ as $a_iX^{\alpha_i}-b_iX^{\beta_i}$ for $i=1,\ldots,n$.
\STATE $M\leftarrow[\alpha_1-\beta_1;\ldots;\alpha_n-\beta_n]$
\STATE compute the Smith Normal Form $PDQ$ of $M$
\STATE $\rho\leftarrow(\delta(b_1/a_1),\ldots,\delta(b_n/a_n))^{P^{-1}}$
\IF {$w\in v(\pi)\Z^n$ and $\rho_i$ is a $d_i$-th power in $k^\ast$ for all $i=1,\ldots,n$}
\IF {${\rm char}(k)\,|\,d_1\cdots d_n$}
\RETURN NO
\ELSE
\STATE $e_i\leftarrow |\{\xi\in k^\ast\,:\,\xi^{d_i}=1\}|$ for $i=1,\ldots,n$
\STATE $s\leftarrow s+\prod_{i=1}^ne_i$
\ENDIF
\ENDIF
\ENDIF
\ENDFOR
\PRINT the system has $s$ solutions in $(K^\ast)^n$
\RETURN YES
\end{algorithmic}
\end{algorithm}

Algorithm~\ref{alg-reg-test} is presented above in pseudo-code with the maximum generality, in order to match the
notation and logic behind Proposition~\ref{reg3}. In any real implementation of the algorithm, the test in line~15
and the formula in line~19, should be replaced by some specific instructions depending on the field $k$.

\begin{itemize}
\item When $k$ is a finite field of cardinality $q=|k|$, the test in line~15 can be rewritten as $\rho_i^{(q-1)/\gcd(q-1,d_i)}=1$,
and line~19 can be replaced by $e_i\leftarrow\gcd(q-1,d_i)$.
\item When $k$ is algebraically closed, line~15 can be simply skipped, since this tests always yields
true, and line~19 becomes $e_i\leftarrow d_i$, or more simply, line~20 becomes $s\leftarrow s+|\det(M)|$ and line~19 is
deleted.
\item When ${\rm char}(k)=0$, the test made in line~16 is not necessary, since by Lemma~\ref{reg2}, the matrix $M$ in line~11
has non-zero determinant, and therefore $d_1\cdots d_n=\det(D)=|\det(M)|\neq 0$.
\end{itemize}

In the univariate case, regular polynomials are very easy to describe. First of all, the tropical hypersurface of a univariate
polynomial is always finite. Moreover, for each $w$ in the tropical set, the lower polynomial $f^{[w]}$ contains all the monomials
$aX^\alpha$ of $f$ such that the point $(\alpha,v(a))$ lies on the lower edge of ${\rm NP}(f)$ with slope $w$. This means that,
in order to have regularity, the lower edges of ${\rm NP}(f)$ must not contain any point (corresponding to a monomial of $f$) other
than the vertices. In addition to this, each lower binomial $f^{[w]}=aX^\alpha+bX^\beta$ must have either $w\not\in v(\pi)\Z$, or
${\rm char}(k)\nmid\alpha-\beta$, or $\delta(b/a)$ not a $(\alpha-\beta)$-th power in $k^\ast$. Compared with the notion of
regularity given in~\cite[Def.~1]{AvIbr}, the definition in this paper includes a broader class of polynomials, while the formula
for the total number of roots in $K^\ast$ provided in~\cite[Thm.~4.4, Thm.~4.5]{AvIbr} is the same as the formula implied by our
Algorithm~\ref{alg-reg-test}.

\bigskip

Consider a set $\mathcal{A}=\{\alpha_1<\cdots<\alpha_t\}\subseteq\Z$ with $t\geq 2$. Denote $K[X]_{\mathcal{A}}$
the set of polynomial supported by $\mathcal{A}$, i.e. $K[X]_{\mathcal{A}}=\{\sum_{\alpha\in\mathcal{A}}a_{\alpha}X^\alpha\,:\,a_{\alpha}\neq 0\}$.
For each $f\in K[X]_{\mathcal{A}}$ we define the support of the Newton Polygon of $f$ as the set
$\mathcal{B}=\{\alpha\in\mathcal{A}\,:\,(\alpha,v(a_\alpha))\in\;\mbox{lower hull of}\;{\rm NP}(f)\}$. The subset of the polynomials
in $K[X]_{\mathcal{A}}$ with Newton Polygon supported at $\mathcal{B}$ is denoted $K[X]_{\mathcal{A}}^{\mathcal{B}}$. Note that we
always have $\{\alpha_1,\alpha_t\}\subseteq\mathcal{B}\subseteq\mathcal{A}$, and that $K[X]_\mathcal{A}=\bigcup_{\{\alpha_1,\alpha_t\}\subseteq\mathcal{B}\subseteq\mathcal{A}}K[X]_{\mathcal{A}}^{\mathcal{B}}$. The discussion
above is summarized in the following corollary.

\begin{coro}\label{np-s}
Let $\mathcal{A}=\{\alpha_1<\cdots<\alpha_t\}\subseteq\Z$ be a set with $t\geq 2$ and ${\rm char}(k)\nmid\alpha_j-\alpha_i$ for all $i\neq j$.
Let $\mathcal{B}=\{\alpha_1=\beta_1<\cdots<\beta_{|\mathcal{B}|}=\alpha_t\}\subseteq\mathcal{A}$ and take
$f=\sum_{\alpha\in\mathcal{A}}a_{\alpha}X^{\alpha}\in K[X]_{\mathcal{A}}^{\mathcal{B}}$. Then
$${\rm Trop}(f)=\left\{ -\frac{v(a_{\beta_{i+1}})-v(a_{\beta_i})}{\beta_{i+1}-\beta_i}\,:\,i=1,\ldots,|\mathcal{B}|-1\right\},$$
i.e. ${\rm Trop}(f)$ is the set of minus the slopes of the segments of ${\rm NP}(f)$. Moreover, $f$ is regular if and only if the points
$\{(\beta,v(a_\beta))\,:\,\beta\in\mathcal{B}\}$ are all vertices of the Newton Polygon, and in this case, the number of roots of $f$ in
$K^\ast$ is equal to
$$\sum_{i=1}^{|\mathcal{B}|-1}\chi_{_{v(\pi)\Z}}\left(\frac{v(a_{\beta_{i+1}})-v(a_{\beta_i})}{\beta_{i+1}-\beta_i}\right)\left|Z_k(\delta(a_{\beta_{i+1}})X^{\beta_{i+1}}+\delta(a_{\beta_{i}})X^{\beta_{i}})\right|.$$
\end{coro}

Finally, note that given a polynomial $f\in K[X]_\mathcal{A}$ and a subset $\{\alpha_1,\alpha_t\}\subseteq\mathcal{B}\subseteq\mathcal{A}$,
it is possible to determine whether $f$ belongs to $K[X]_{\mathcal{A}}^{\mathcal{B}}$ by just testing a few linear inequalities in the
valuations of the coefficients: a point $\alpha\in\mathcal{A}$ is in the support of the Newton Polygon if and only if
$$v(a_\alpha)\leq v(a_{\alpha'})\frac{\alpha-\alpha''}{\alpha'-\alpha''} + v(a_{\alpha''})\frac{\alpha'-\alpha}{\alpha'-\alpha''}$$
for all $\alpha',\alpha''\in\mathcal{A}$ with $\alpha'<\alpha<\alpha''$. Inspired by this simple test, we introduce the set $S(\mathcal{B}/\mathcal{A})\subseteq\R^t$ defined as the set of all vectors $(v_1,\ldots,v_t)\in\R^t$ such that
$$v_i\leq v_j\frac{\alpha_i-\alpha_k}{\alpha_j-\alpha_k} + v_k\frac{\alpha_j-\alpha_i}{\alpha_j-\alpha_k}$$
for all $1\leq j<i<k\leq t$ if and only if $\alpha_i\in\mathcal{B}$. This means that a polynomial $f\in K[X]_{\mathcal{A}}$ belongs to
$K[X]_{\mathcal{A}}^{\mathcal{B}}$ if and only if $(v(a_{\alpha_1}),\ldots,v(a_{\alpha_t}))\in S(\mathcal{B}/\mathcal{A})$. In the analysis
of random univariate polynomials of Section~\ref{sec-prob}, we will need the Lebesgue measure of the set $S(\mathcal{B}/\mathcal{A})\cap [0,1]^t$,
which will be denoted $P(\mathcal{B}/\mathcal{A})$. Roughly speaking, $P(\mathcal{B}/\mathcal{A})$ is the probability that the
set of points $\{(\alpha_1,v_1),\ldots,(\alpha_t,v_t)\}$, where $v_i\sim\mathcal{U}[0,1]$ are independent random
variables, has Newton Polygon supported at $\mathcal{B}$.
From the form of the equations defining these sets, note that $(v_1,\ldots,v_t)\in S(\mathcal{B}/\mathcal{A})$ if and only if
$(av_1+b,\ldots,av_t+b)\in S(\mathcal{B}/\mathcal{A})$ for all $a,b\in\R$, i.e. these sets are invariant under rescaling and translations.
In particular, the measure of $S(\mathcal{B}/\mathcal{A})\cap [a,b]^t$ is equal to $(b-a)^t P(\mathcal{B}/\mathcal{A})$.

\section{Tropical genericity of regular systems}\label{sec-gen}

\begin{defn}\label{def-gen1}
Consider a proposition $P:(K^\ast)^n\to\{{\rm True},{\rm False}\}$. We say that $P$ is true for any generic $x\in(K^\ast)^n$ if and only
if $P^{-1}({\rm False})$ is contained in an algebraic hypersurface of $(K^\ast)^n$. Similarly,
a proposition $P:(K^\ast)^n\to\{{\rm True},{\rm False}\}$ is said to be true for any tropically generic $x\in(K^\ast)^n$ if and only if
$v(P^{-1}({\rm False}))$ is contained in a finite union of hyperplanes of $\R^n$.
\end{defn}

Note that genericity implies tropical genericity: if a statement $P$ is true for generic $x\in(K^\ast)^n$, then there is a hypersurface
$Z_K(G)\subseteq(K^\ast)^n$ that contains $P^{-1}({\rm False})$, and therefore, the tropical hypersurface ${\rm Trop}(G)$, which is contained in a
finite union of hyperplanes of $\R^n$, contains $v(P^{-1}({\rm False}))$.

\bigskip

Let $\mathcal{A}_1,\ldots,\mathcal{A}_n\subseteq\Z^n$ be nonempty finite sets. Consider a system of polynomials $F=(f_1,\ldots,f_n)$ in $K\vars$ with
undetermined (non-zero) coefficients and ${\rm Supp}(f_i)=\mathcal{A}_i$ for all $i=1,\ldots,n$. Let $N=|\mathcal{A}_1|+\cdots+|\mathcal{A}_n|$
be the number of coefficients in $F$. Once these supports have been fixed, we can speak about propositions for generic or tropically generic
systems $F$ in the sense of Definition~\ref{def-gen1}: the domain of the propositions is understood to be the coefficient space $(K^\ast)^N$ of
the systems.

\begin{thm}\label{gen1}
Any tropically generic system $F=(f_1,\ldots,f_n)$ in $K\vars$ has finite tropical prevariety ${\rm Trop}(F)$ and its lower polynomials
$f_i^{[w]}$ are binomials for all $w\in{\rm Trop}(F)$ and $i=1,\ldots,n$.
\end{thm}
\begin{proof}
Write $f_i=\sum_{\alpha\in\mathcal{A}_i}a_{\alpha}^{(i)}X^\alpha$ for $i=1,\ldots,n$. Assume first that ${\rm Trop}(F)$ is an infinite set.
We will show that the vector $\mu=v(a_\alpha^{(i)})_{1\leq i\leq n,\,\alpha\in\mathcal{A}_i}\in\R^N$ lies on a finite union
of hyperplanes $H\subseteq\R^N$ that depends only on the sets $\mathcal{A}_1,\ldots,\mathcal{A}_n$. By Lemma~\ref{semi2}, there are $\alpha_i,\beta_i\in\mathcal{A}_i$
for $i=1,\ldots,n$ such that the system of linear equations
\begin{eqnarray}\label{eq-xxxx}
v(a_{\alpha_1}^{(1)}) + \alpha_1\cdot w & = & v(a_{\beta_1}^{(1)}) + \beta_1\cdot w \nonumber \\
\vdots\qquad\quad & & \quad\qquad\vdots \\
v(a_{\alpha_n}^{(n)}) + \alpha_n\cdot w & = & v(a_{\beta_n}^{(n)}) + \beta_n\cdot w \nonumber
\end{eqnarray}
has infinitely many solutions $w\in\R^n$. This means that the determinant of the matrix whose rows are $\alpha_i-\beta_i$ for $i=1,\ldots,n$ is zero and
that
$$\left(v(a_{\alpha_1}^{(1)})-v(a_{\beta_1}^{(1)}),\ldots,v(a_{\alpha_n}^{(n)})-v(a_{\beta_n}^{(n)})\right)\in
   \big\langle\alpha_i-\beta_i\,:\,i=1,\ldots,n\big\rangle$$
Since the vectors $\alpha_i-\beta_i$ for $i=1,\ldots,n$ are $\R$-linearly dependent (the determinant of the matrix is zero), the subspace at the right
side of the condition above has codimension one (or more) in $\R^n$. This translates into a condition that says that $\mu$ belongs to some hyperplane
of $\R^N$ that depends only on $\alpha_1,\beta_1,\ldots,\alpha_n,\beta_n$. We conclude by taking $H$ as the union of these hyperplanes for all possible
choice of $\alpha_1,\beta_1\in\mathcal{A}_1,\ldots,\alpha_n,\beta_n\in\mathcal{A}_n$ such that $\{\alpha_i-\beta_i\,:\,i=1,\ldots,n\}$ is a $\R$-linearly
dependent set.

Now assume that $\mu\not\in H$, and in particular ${\rm Trop}(F)$ is finite, but $f_i^{[w]}$ has three or more terms for some $i=1,\ldots,n$ and
$w\in{\rm Trop}(F)$. We will show that there is a finite union  of hyperplanes $H'\subseteq\R^N$, that depends only on $\mathcal{A}_1,\ldots,\mathcal{A}_n$,
such that $\mu\in H'$. It is enough to consider the case where the polynomial with three or more monomials is $f_1^{[w]}$. The point $w$ is the
unique solution of the system (\ref{eq-xxxx}) for some $\alpha_1,\beta_1\in\mathcal{A}_1,\ldots,\alpha_n,\beta_n\in\mathcal{A}_n$, and in particular,
the monomials $a_{\alpha_1}^{(1)}X^{\alpha_1}$ and $a_{\beta_1}^{(1)}X^{\beta_1}$ are in $f_1^{[w]}$. Since $f_1^{[w]}$ has three or more terms, there
exists $\gamma_1\in\mathcal{A}_1\setminus\{\alpha_1,\beta_1\}$ such that the term $a_{\gamma_1}^{(1)}X^{\gamma_1}$ is in $f_1^{[w]}$. The equation
${v(a_{\gamma_1}^{(1)})+\gamma_1\cdot w=v(a_{\alpha_1}^{(1)})+\alpha_1\cdot w}$, where $w$ is the unique solution of (\ref{eq-xxxx}) expressed as a
linear function of $v(a_{\alpha_1}^{(1)}),v(a_{\beta_1}^{(1)}),\ldots,v(a_{\alpha_n}^{(n)}),v(a_{\beta_n}^{(n)})$ gives a non-trivial linear equation
for the valuation of the coefficients of $F$, thus restricting $\mu$ to a hyperplane (that depends only on the choice of $\alpha_1,\beta_1,\gamma_1,\ldots,\alpha_n,\beta_n$). We conclude by taking the union of all these possible hyperplanes.
\end{proof}

According to Algorithm~\ref{alg-reg-test}, a system $F$ can fail to be regular for three different reasons (see lines 3, 8 and 17): when
the tropical prevariety is not finite, when some lower polynomial has more than two terms, or when ${\rm char}(k)$ divides the determinant
of certain invertible matrices. By Theorem~\ref{gen1}, the first two do not occur for tropically generic systems, and in particular, if
${\rm char}(k)=0$, then any tropically generic system is regular. The same idea works for any characteristic coprime to all the determinants
that can arise in the test in line~16.

\begin{coro}\label{gen2}
Let $\mathcal{A}_1,\ldots,\mathcal{A}_n\subseteq\Z^n$ be nonempty finite sets. Assume that ${\rm char}(k)=0$ or that ${\rm char}(k)$ is coprime to
the determinant of all invertible matrices $M=[\alpha_1-\beta_1;\ldots;\alpha_n-\beta_n]$ with $\alpha_i,\beta_i\in\mathcal{A}_i$ for
$i=1,\ldots,n$. Then, any tropically generic system of polynomials $F=(f_1,\ldots,f_n)$ in $K\vars$ with ${\rm supp}(f_i)=\mathcal{A}_i$ for
$i=1,\ldots,n$ is regular.
\end{coro}

\section{The expected number of roots of a random polynomial}\label{sec-prob}

In this section we restrict the discussion to univariate polynomials. Let $\mathcal{A}\subseteq\Z$ be a nonempty finite set. Assume that the
characteristic of the residue field is either zero or coprime to $\alpha-\beta$ for all $\alpha,\beta\in\mathcal{A}$ with ${\alpha\neq\beta}$. This assumption
ensures, by Corollary~\ref{gen2}, that any tropically generic polynomial $f\in K[X]$ with support $\mathcal{A}$ is regular. Since we have an explicit formula for
the number of roots of regular polynomials in $K^\ast$, we should be able to obtain the expected number of roots of $f$ in $K^\ast$,
provided that we select the coefficients of $f$ at random with a distribution that produces tropically generic polynomials with probability $1$.

\bigskip

Let $D_1$ be a probability distribution on $k^\ast$, let $D_2$ be a probability distribution on $1+\M$, and let $M>0$. We will select
elements in $K^\ast$ at random by selecting their valuation uniformly in $[-M,M]\cap v(\pi)\Z$, their first digit according to $D_1$, and
their tail in $1+\M$ according to $D_2$. This procedure induces a probability distribution in $K^\ast$ that extends to a distribution in
$K[X]_{\mathcal{A}}=\{f\in K[X]\,:\,{\rm supp}(f)=\mathcal{A}\}$. Denote by $E(\mathcal{A},D_1,D_2,M,K)$ the expected number of roots
in $K^\ast$ of a polynomial $f\in K[X]_{\mathcal{A}}$ chosen at random with this distribution. Since the number of roots of these
polynomials can not exceed their degree, we have that $E(\mathcal{A},D_1,D_2,M,K)\leq\max_{\alpha,\beta\in\mathcal{A}}|\alpha-\beta|$.
The main goal of this section is to find the value of
$$E(\mathcal{A},D_1,D_2,K)=\lim_{M\to\infty}E(\mathcal{A},D_1,D_2,M,K)$$
for several fields $K$ and probability distributions $D_1$ and $D_2$.

\begin{lemma}\label{prob0}
Consider the probability distribution in $K[X]_{\mathcal{A}}$ induced by $D_1$, $D_2$, and $M>0$.
Assume that ${\rm char}(k)$ is zero or coprime to $\alpha-\beta$ for all $\alpha,\beta\in\mathcal{A}$ with $\alpha\neq\beta$.
Then, the probability that a random $f\in K[X]_{\mathcal{A}}$ is not regular approches zero as $M\to\infty$.
\end{lemma}
\begin{proof}
Let $t=|\mathcal{A}|$ and $\mathcal{A}=\{\alpha_1,\ldots,\alpha_t\}$. By Corollary~\ref{gen2}, there are hyperplanes
$H_1,\ldots,H_n\subseteq\R^t$ such that any polynomial $f=\sum_{i=1}^ta_i X^{\alpha_i}$ with $(v(a_1),\ldots,v(a_t))\not\in\cup_{j=1}^nH_j$
is regular. In particular, it is enough to show that the probability that a random $f\in K[X]_{\mathcal{A}}$ has coefficients
with valuation in $\cup_{j=1}^nH_j$ goes to zero as $M\to\infty$. Note that this probability does not depend on the distributions $D_1$
and $D_2$. Moreover, since the valuation of the coefficients is selected at random in the box $([-M,M]\cap v(\pi)\Z)^t$ with uniform
distribution, the probability of being in the union of $n$ hyperplanes is less than or equal to $n/(2[M/v(\pi)]+1)$ (each hyperplane contains at most
$1/(2[M/v(\pi)]+1)$ of the points in the box). As the size of the box increases, this probability approaches zero.
\end{proof}

By Lemma~\ref{prob0}, the probability that a random $f\in K[X]_{\mathcal{A}}$ is regular approaches $1$ as $M$ goes to infinity. Besides,
we have shown in Proposition~\ref{reg3} (or in Algorithm~\ref{alg-reg-test}) that the number of solutions of a regular system does not depend on the
tail of the coefficients. In particular, the value of $E(\mathcal{A},D_1,D_2,K)$ does not depend on $D_2$, and for this reason, it will
be simply written as $E(\mathcal{A},D_1,K)$.

\bigskip

Before stating the main result, we need to fix some notation. For any $\gamma\in\N$, we denote by $E_k(\gamma,D_1)$ the expected
number of roots in $k^\ast$ of the binomial $aX^\gamma+b$ with coefficients $a,b\in k^\ast$ chosen at random (independently)
according to the distribution $D_1$. For instance, when $k$ is algebraically closed, we have $E_k(\gamma,D_1)=\gamma$ regardless
of the distribution $D_1$. If $k$ is a finite field and $D_1$ is the uniform distribution in $k^\ast$, then $E_k(\gamma,D_1)=1$, since
the number of roots of $X^\gamma=-b/a$ is either zero or $\gcd(|k^\ast|,\gamma)$, and the latter happens only when $-b/a$ is a
$\gamma$-th power in $k^\ast$ which occurs with probability $1/\gcd(|k^\ast|,\gamma)$. A similar situation arises in the case
$k=\R$ and $D_1$ a distribution such that $\R_{>0}$ and $\R_{<0}$ have each probability $1/2$: if $\gamma$ is odd, then $X^\gamma=-b/a$
has always one real root, and if $\gamma$ is even, the number of roots is either $0$ or $2$ depending on whether ${\rm sgn}(ab)$ is~$1$
or~$-1$, but in both cases we have $E_k(\gamma,D_1)=1$.

\medskip

\begin{center}
  \begin{tabular}{|c|c|c|}
     \hline
     $k$ & $D_1$ & $E_k(\gamma,D_1)$ \\
     \hline
     $k$ alg.~closed & any & $\gamma$ \\
     \hline
     $k=\R$ & ${\rm Prob}(\R_{<0})={\rm Prob}(\R_{>0})=1/2$ & $1$ \\
     \hline
     $|k|<+\infty$ & uniform in $k^\ast$ & $1$ \\
     \hline
  \end{tabular}
\end{center}

\medskip

\begin{thm}\label{prob1}
Let $\mathcal{A}=\{\alpha_1<\alpha_2\cdots<\alpha_t\}\subseteq\Z$ finite with $t\geq 2$.
Assume that ${\rm char}(k)=0$ or that ${\rm char}(k)$ is coprime to $\alpha-\beta$ for any pair
of elements $\alpha,\beta\in\mathcal{A}$ with $\alpha\neq\beta$. Let $D_1$ be a probability distribution in $k^\ast$. Then
$$E(\mathcal{A},D_1,K)=\!\!\!\!\!\!\!\!\sum_{\{\alpha_1,\alpha_t\}\subseteq\mathcal{B}\subseteq\mathcal{A}}\!\!\!\!\!\!\!\!
P(\mathcal{B}/\mathcal{A})\sum_{i=1}^{|\mathcal{B}|-1}\frac{E_k(\beta_{i+1}-\beta_i,D_1)}{\beta_{i+1}-\beta_i}$$
where $\mathcal{B}=\{\alpha_1=\beta_1<\beta_2<\cdots<\beta_{|\mathcal{B}|}=\alpha_t\}$.
\end{thm}
\begin{proof}
Let $D_2$ be any probability distribution in $1+\M$ and let $M>0$. Random polynomials $f=\sum_{\alpha\in\mathcal{A}}a_\alpha X^{\alpha}\in K[X]_{\mathcal{A}}$
will be chosen according to $D_1$, $D_2$ and $M$. For each subset
$\mathcal{B}=\{\alpha_1=\beta_1<\cdots<\beta_{|\mathcal{B}|}=\alpha_t\}\subseteq\mathcal{A}$,
denote by $K[X]_{\mathcal{A}}^{\mathcal{B}}$ the set of polynomials
$f\in K[X]_{\mathcal{A}}$ with Newton Polygon supported at~$\mathcal{B}$. By definition, we have
that $$E(\mathcal{A},D_1,D_2,M,K)=\int_{K[X]_{\mathcal{A}}}|Z_K(f)|\,df=
\sum_{\genfrac{}{}{0pt}{}{{\{\alpha_1,\alpha_t}\}\subseteq\mathcal{B}}{\mathcal{B}\subseteq\mathcal{A}}}\int_{K[X]_{\mathcal{A}}^{\mathcal{B}}}|Z_K(f)|\,df$$
and also
$$E(\mathcal{A},D_1,D_2,K)=\sum_{\genfrac{}{}{0pt}{}{{\{\alpha_1,\alpha_t}\}\subseteq\mathcal{B}}{\mathcal{B}\subseteq\mathcal{A}}}\lim_{M\to\infty}\int_{K[X]_{\mathcal{A}}^{\mathcal{B}}}|Z_K(f)|\,df.$$
For any $f\in K[X]_{\mathcal{A}}^{\mathcal{B}}$, define $$N(f)=\sum_{i=1}^{|{\mathcal{B}}|-1}\chi_{_{v(\pi)\Z}}\left(\frac{v(a_{\beta_{i+1}})-v(a_{\beta_i})}{\beta_{i+1}-\beta_i}\right)\left|Z_k(\delta(a_{\beta_{i+1}})X^{\beta_{i+1}}+\delta(a_{\beta_{i}})X^{\beta_{i}})\right|,$$
where $\chi_S(\cdot)$ represents the characteristic function of the set $S$.
This gives a function $N:K[X]_{\mathcal{A}}\to\N_0$ that, by Proposition~\ref{reg3}, coincides with $|Z_K(f)|$ for any $f\in K[X]_{\mathcal{A}}$ regular.
Moreover, the difference $N(f)-|Z_K(f)|$ is bounded on $K[X]_{\mathcal{A}}$. By Theorem~\ref{gen1}, the probability of the set of non-regular polynomials
approaches $0$ as $M$ goes to infinity, and then we can also write
\begin{align*}
E(\mathcal{A},D_1,D_2,K)&= \sum_{\genfrac{}{}{0pt}{}{{\{\alpha_1,\alpha_t}\}\subseteq\mathcal{B}}{\mathcal{B}\subseteq\mathcal{A}}}\lim_{M\to\infty}\int_{K[X]_{\mathcal{A}}^{\mathcal{B}}}N(f)\,df= \\
&=\sum_{\genfrac{}{}{0pt}{}{{\{\alpha_1,\alpha_t}\}\subseteq\mathcal{B}}{\mathcal{B}\subseteq\mathcal{A}}}\lim_{M\to\infty}\int_{K[X]_{\mathcal{A}}}\chi_{_{K[X]_{\mathcal{A}}^{\mathcal{B}}}}(f)N(f)\,df=\\
&=\sum_{\genfrac{}{}{0pt}{}{{\{\alpha_1,\alpha_t}\}\subseteq\mathcal{B}}{\mathcal{B}\subseteq\mathcal{A}}}\sum_{i=1}^{|{\mathcal{B}}|-1}\lim_{M\to\infty}\int_{K[X]_{\mathcal{A}}}N_{\mathcal{B},i}(f)\,df
\end{align*}
where $N_{\mathcal{B},i}(f)$ is the expression $$\chi_{_{K[X]_{\mathcal{A}}^{\mathcal{B}}}}(f)\chi_{_{v(\pi)\Z}}\left(\frac{v(a_{\beta_{i+1}})-v(a_{\beta_i})}{\beta_{i+1}-\beta_i}\right)\left|Z_k(\delta(a_{\beta_{i+1}})X^{\beta_{i+1}}+\delta(a_{\beta_{i}})X^{\beta_{i}})\right|.$$
Any polynomial $f\in K[X]_{\mathcal{A}}$ correspond with a unique point
$(w,\delta,e)\in ([-M,M]\cap v(\pi)\Z)^t\times (k^\ast)^t\times (1+\M)^t$.
Using this representation, we can write $\int_{K[X]_{\mathcal{A}}}N_{\mathcal{B},i}(f)\,df$ as the triple integral
\begin{equation*}
\int\!\!\!\!\int\!\!\!\!\int\!\chi_{_{K[X]_{\mathcal{A}}^{\mathcal{B}}}}(f)\chi_{_{v(\pi)\Z}}\left(\frac{w_{\beta_{i+1}}-w_{\beta_{i}}}{\beta_{i+1}-\beta_i}\right)\left|Z_k(\delta_{\beta_{i+1}}X^{\beta_{i+1}}+\delta_{\beta_{i}}X^{\beta_{i}})\right|de\,d\delta\,dw.
\end{equation*}
Since the function
$\chi_{_{K[X]_{\mathcal{A}}^{\mathcal{B}}}}(f)\chi_{_{v(\pi)\Z}}((w_{\beta_{i+1}}-w_{\beta_{i}})/
(\beta_{i+1}-\beta_i))$
depends only on $w$, and the function $|Z_k(\delta_{\beta_{i+1}}X^{\beta_{i+1}}+\delta_{\beta_{i}}X^{\beta_{i}})|$ depends only on $\delta$, the triple integral above can be splitted as a product of three simple integrals. More precisely, we have that $\int_{K[X]_{\mathcal{A}}}N_{\mathcal{B},i}(f)\,df=I_wI_\delta I_e$,
where
\begin{eqnarray*}
I_w & = & \int_{([-M,M]\cap v(\pi)\Z)^t}\chi_{_{K[X]_{\mathcal{A}}^{\mathcal{B}}}}(f)\chi_{_{v(\pi)\Z}}\left(\frac{w_{\beta_{i+1}}-w_{\beta_{i}}}{\beta_{i+1}-\beta_i}\right)\,dw,\\
I_\delta & = & \int_{(k^\ast)^t}\left|Z_k(\delta_{\beta_{i+1}}X^{\beta_{i+1}}+\delta_{\beta_{i}}X^{\beta_{i}})\right|\,d\delta, \\
I_e & = & \int_{(1+\M)^t} 1\,de.
\end{eqnarray*}
It is clear that $I_e=1$ and also $I_\delta=E_k(\beta_{i+1}-\beta_i,D_1)$ by definition.
The integral defining $I_w$ is in fact a finite sum over a lattice: if we write $N=[M/v(\pi)]$ and $v_\alpha=w_\alpha/v(\pi)$, then
\begin{eqnarray*}
I_w & = & (2N+1)^{-t}\sum_{-N\leq v_1,\ldots,v_t\leq N}\chi_{_{K[X]_{\mathcal{A}}^{\mathcal{B}}}}(f)\chi_{_\Z}\left(\frac{v_{\beta_{i+1}}-v_{\beta_{i}}}{\beta_{i+1}-\beta_i}\right)= \\
& = & (2N+1)^{-t}\sum_{\genfrac{}{}{0pt}{}{{-N\leq v_1,\ldots,v_t\leq N}}{ {\beta_{i+1}-\beta_i|v_{\beta_{i+1}}-v_{\beta_i}}}}\chi_{_{K[X]_{\mathcal{A}}^{\mathcal{B}}}}(f).
\end{eqnarray*}
The expression $\chi_{{K[X]_{\mathcal{A}}^{\mathcal{B}}}}(f)$ in the last sum is a function of $v_{\alpha_1},\ldots,v_{\alpha_t}$ that test
whether the Newton Polygon of the set of points $\{(v_\alpha,\alpha):\alpha\in\mathcal{A}\}$ is supported at $\mathcal{B}$, i.e. is equal to
$\chi_{_{S(\mathcal{B}/\mathcal{A})}}(v_{\alpha_1},\ldots,v_{\alpha_t})$. Since the set $S(\mathcal{B}/\mathcal{A})$ is invariant under rescaling
and translations, then
$$I_w=(2N+1)^{-t}\sum_{\genfrac{}{}{0pt}{}{{-N\leq v_1,\ldots,v_t\leq N}}{ {\beta_{i+1}-\beta_i|v_{\beta_{i+1}}-v_{\beta_i}}}}
\chi_{_{S(\mathcal{B}/\mathcal{A})}}\left(\frac{N+v_{\alpha_1}}{2N+1},\ldots,\frac{N+v_{\alpha_t}}{2N+1}\right).$$
Without the condition $\beta_{i+1}-\beta_i|v_{\beta_{i+1}}-v_{\beta_i}$, the expression is exactly a Riemman sum of $\chi_{_{S(\mathcal{B}/\mathcal{A})}}$,
with a partition of $[0,1]^t$ corresponding to the lattice $\{0,1/(2N+1),\ldots,1\}^t$. Adding this extra condition is equivalent to
taking a sublattice of order $\beta_{i+1}-\beta_i$, so $\lim_{M\to\infty}I_w=P(\mathcal{B}/\mathcal{A})(\beta_{i+1}-\beta_i)^{-1}$.
This shows that $$\lim_{M\to\infty}\int_{K[X]_{\mathcal{A}}}N_{\mathcal{B},i}(f)\,df=P(\mathcal{B}/\mathcal{A})\frac{E_k(\beta_{i+1}-\beta_i,D_1)}{\beta_{i+1}-\beta_i}.$$
Going back to our formula for $E(\mathcal{A},D_1,D_2,K)$, we get
$$E(\mathcal{A},D_1,D_2,K)=\!\!\!\!\!\!\!\!\sum_{\{\alpha_1,\alpha_{|\mathcal{A}|}\}\subseteq\mathcal{B}\subseteq\mathcal{A}}\!\!\!\!\!\!\!\!
P(\mathcal{B}/\mathcal{A})\sum_{i=1}^{|\mathcal{B}|-1}\frac{E_k(\beta_{i+1}-\beta_i,D_1)}{\beta_{i+1}-\beta_i}.$$
To conclude the proof, note that the right term does not depend on the probability distribution $D_2$, and then we can safely write $E(\mathcal{A},D_1,K)$,
as claimed.
\end{proof}

We conclude this section with an analysis of the case where the residue field is algebraically closed. In this case, we have
$E_k(\gamma,D_1)=\gamma$, regardless of the probability distribution $D_1$, so the formula of Theorem~\ref{prob1} reduces to
$$E(\mathcal{A},D_1,K)=\!\!\!\!\!\!\!\!\sum_{\{\alpha_1,\alpha_{|\mathcal{A}|}\}\subseteq\mathcal{B}\subseteq\mathcal{A}}\!\!\!\!\!\!\!\!
P(\mathcal{B}/\mathcal{A})(|\mathcal{B}|-1)=1+\sum_{i=2}^tP_i,$$
where $P_i=\sum_{\{\alpha_1,\alpha_i,\alpha_{|\mathcal{A}|}\}\subseteq\mathcal{B}\subseteq\mathcal{A}}P(\mathcal{B}/\mathcal{A})$ is
the probability that $\alpha_i$ is in the support of the Newton Polygon. The value of $P_i$ can be written in terms of integrals, as
shown in the following formula:
$$P_i=\int_0^1\cdots\int_0^1\min_{1\leq j<i<k\leq t}
\left(v_j\frac{\alpha_i-\alpha_k}{\alpha_j-\alpha_k}+v_k\frac{\alpha_j-\alpha_i}{\alpha_j-\alpha_k}\right)dv_1\cdots\widehat{dv_i}\cdots dv_t.$$
The estimations
$$P_i\leq\int_0^1\cdots\int_0^1\max(\min(v_1,\ldots,v_{i-1}),\min(v_{i+1},\ldots,v_t))dv_1\cdots\widehat{dv_i}\cdots dv_t,$$
$$P_i\geq\int_0^1\cdots\int_0^1\min(v_1,\ldots,\widehat{v_i},\ldots,v_t)dv_1\cdots\widehat{dv_i}\cdots dv_t,$$
show that $\frac{1}{t}\leq P_i\leq\frac{1}{i}+\frac{1}{t-i+1}-\frac{1}{t}$, and therefore
$$2-\frac{2}{t}\leq E(\mathcal{A},D_1,K)\leq 2\sum_{i=2}^t\frac{1}{i}\leq 2\ln(t).$$

\section*{Acknowledgements}

We would like to thank Maurice Rojas and Bernd Sturmfels for several fruitful discussions about regularity
and semiregularity, and for encouraging us to publish these results.

\end{document}